\documentclass{mcom-l}

\usepackage{url}
\usepackage{amsthm}
\usepackage{amsmath}
\usepackage{amssymb}
\newtheorem{Lem}{Lemma}

\newtheorem{Cor}{Corollary}
\newtheorem{theorem}{Theorem}

\numberwithin{equation}{section}
\begin{document}

\title[An improved upper bound for the argument of $\zeta(\frac{1}{2} +it)$ II]{An improved upper bound for the argument of the Riemann zeta-function on the critical line II}

\author{Timothy Trudgian}
\address{Mathematical Sciences Institute, The Australian National University, ACT 0200, Australia}
\email{timothy.trudgian@anu.edu.au}
\thanks {Supported by ARC Grant DE120100173.}


\date{February 24, 2011.}


\begin{abstract}
This paper concerns the function $S(T)$, the argument of the Riemann zeta-function along the critical line. The main result is that
\begin{equation*}
|S(T)| \leq 0.111\log T + 0.275\log \log T + 2.450,
\end{equation*}
which holds for all $T\geq e$.
\end{abstract}

\maketitle

\section{Summary of Results}
This paper is the sequel to \cite{TrudS} and is related to \cite{TrudDed}; reference will be made frequently to these papers. Write
\begin{equation}\label{opgam}
|S(T)| \leq a \log T + b\log \log T + c, \quad\quad \textrm{for } T\geq T_{0},
\end{equation}
whence the following table provides a brief historical summary.
\begin{table}[h]
\caption{Bounds on $S(T)$ in (\ref{opgam})}
\label{opp}
\centering
\begin{tabular}{l c c c c}
\hline\hline
& $a$ & $b$ & $c$ & $T_{0}$ \\[0.5 ex] \hline
Von Mangoldt \cite{vonMangoldt} 1905 & 0.432 & 1.917 & 12.204 & 28.588\\
Grossmann \cite{Grossmann} 1913 & 0.291 & 1.787 & 6.137 & 50\\
Backlund \cite{Backlund} 1914 & 0.275 & 0.979 & 7.446 & 200\\
Backlund \cite {Backlund1918} 1918 & 0.137 & 0.443 & 4.35 & 200\\
Rosser \cite{Rossers} 1941 & 0.137 & 0.443 & 1.588 & 1467\\
Trudgian \cite{TrudS} 2012 & 0.17 & 0 & 1.998 & $e$\\
Trudgian (Theorem \ref{Onenad}) 2012 & 0.111 & 0.275 & 2.450 & $e$
\\ \hline\hline
\end{tabular}
\end{table}

Note that the result in \cite{TrudS} improves on that in \cite{Rossers} when $25 \leq T \leq 10^{15}$. The purpose of this article is to improve on the result in \cite{Rossers} for all $T$. This is achieved with the following theorem.

\begin{theorem}\label{Onenad}
\begin{equation*}
|S(T)| \leq 0.111 \log T + 0.275 \log\log T + 2.450,
\end{equation*}
for $T \geq e$.
\end{theorem}

This implies the following result concerning $N(T)$, the number of complex zeroes of $\zeta(s)$ with imaginary parts in $(0, T)$.

\begin{Cor}\label{onlycor}
For $T\geq T_{0}\geq e$
\begin{equation*}
\bigg|N(T) - \frac{T}{2\pi}\log\frac{T}{2\pi e } -\frac{7}{8}\bigg| \leq 0.111 \log T + 0.275 \log\log T + 2.450 + \frac{0.2}{T_{0}}.
\end{equation*}
\end{Cor}

It is known (see, e.g., \cite{Brent, terinter}) that\footnote{Indeed, (\ref{Odl}) is the statement that Gram's Law holds for all $0\leq T \leq 280$ and that Rosser's Rule holds for all $0\leq T \leq 6.8\cdot 10^{6}$.}
\begin{equation}\label{Odl}
|S(T)| \leq 1, \textrm{for } 0\leq T \leq 280, \quad |S(T)|\leq 2, \textrm{for } 0\leq T\leq 6.8 \cdot10^6.
\end{equation}
The approach taken in this paper is to prove results initially for $T>T_{0}>6.8 \cdot 10^{6}$, and then for all $T$ by using (\ref{Odl}). Note that Theorem \ref{Onenad} is sharper than Rosser's bound in \cite{Rossers} whenever $T\geq 6.4 \cdot 10^{6}$; for smaller values of $T$ one is better placed using (\ref{Odl}), which is superior to both Theorem \ref{Onenad} and the bound in\cite{Rossers}.

Explicit bounds on $S(T)$ are used in conjunction with the verification of the Riemann hypothesis to a certain height. Hence there is some interest in obtaining, not necessary the smallest coefficient of $\log T$ in Theorem \ref{Onenad}, but good bounds of the form $|S(T)|\leq \alpha\log T$ for all $T\geq T_{0}$, where $T_{0}$ is the point up to which the Riemann hypothesis has been verified. From (\ref{opgam}) one can write
\begin{equation*}
|S(T)|\leq \log T \left(a + \frac{b\log\log T_{0}}{\log T_{0}} + \frac{c}{\log T_{0}}\right) = \alpha\log T,
\end{equation*}
for $T\geq T_{0}\geq e^{e}$. Values of $\alpha$ have been provided in the following table. The parameters $\eta$ and $r$ are those found in (\ref{Sfinal}).

\begin{table}[h]
\caption{Bounds on $|S(T)|\leq \alpha\log T$ for $T\geq T_{0}$}
\label{opp2}
\centering
\begin{tabular}{c c c c}
\hline\hline
$T_{0}$ & $\alpha$ & $\eta$ & $r$ \\[0.5 ex] \hline
$10^{6}$ & 0.260 & 0.28 & 2.28 \\
$10^{7}$ & 0.246 & 0.24 & 2.35 \\
$10^{8}$ & 0.235 & 0.22 & 2.38 \\
$10^{9}$ & 0.226 & 0.19 & 2.44 \\
$10^{10}$ & 0.218 & 0.17 & 2.49 \\
$10^{11}$ & 0.212 & 0.16 & 2.51  \\
$10^{12}$ & 0.206 & 0.15 & 2.53 \\
$10^{13}$ & 0.202 & 0.14 & 2.51  \\
$10^{14}$ & 0.197 & 0.13 & 2.50 \\
$10^{15}$ & 0.193 & 0.12 & 2.48 \\
 \hline\hline
\end{tabular}
\end{table}

Throughout the paper $\eta$ denotes a parameter satisfying $0<\eta\leq \frac{1}{2}$.

\section{Introduction}
Write $a(s) = (s-1)\zeta(s)$, whence $a(s)$ is an entire function. 
The function 
\begin{equation*}\label{xi1}
\xi(s) = \tfrac{1}{2} s(s-1)\pi^{-s/2}\Gamma(\tfrac{s}{2})\zeta(s) = \tfrac{1}{2} s\pi^{-s/2}\Gamma(\tfrac{s}{2})a(s)
\end{equation*}
is entire and satisfies the functional equation
\begin{equation}\label{xif}
\xi(s) = \xi(1-s).
\end{equation}
Let $N(T)$ denote the number of zeroes $\rho = \beta + i\gamma$ of $\zeta(s)$ for which $0< \beta <1$ and $0< \gamma <T$. For any $\sigma_{1}\in(1, 2]$ form a rectangle with vertices at $\sigma_{1} \pm iT$ and $1-\sigma_{1} \pm iT$. Let $\mathcal{C}$ denote the portion of the rectangle in the region $\Re(s) \geq \frac{1}{2}$ and $\Im(s)\geq 0$. Write $\mathcal{C}$ as the union of two straight lines, viz.\ let $\mathcal{C} = \mathcal{C}_{1} + \mathcal{C}_{2}$, where $\mathcal{C}_{1}$ connects $0$ to $\sigma_{1} + iT$ and $\mathcal{C}_{2}$ connects $\sigma_{1} + iT$ to $\frac{1}{2} + iT$.   From Cauchy's theorem and (\ref{xif})  one deduces that $\pi N(T) =  \Delta_{\mathcal{C}} \arg \xi(s).$ Thus
\begin{equation}\label{N2}
\begin{split}
\pi N(T) &= \Delta_{\mathcal{C}}\arg\pi^{-s/2} + \Delta_{\mathcal{C}}\arg s(s-1) + \Delta_{\mathcal{C}}\arg\Gamma\left(\frac{s}{2}\right)  + \pi S(T),
\end{split}
\end{equation}
where
\begin{equation}\label{sdeee}
\pi S(T) = \Delta_{\mathcal{C}_{1}}\arg \zeta(s)+\Delta_{\mathcal{C}_{2}}\arg a(s) - \Delta_{\mathcal{C}_{2}}\arg (s-1).
\end{equation}
The only terms in (\ref{N2}) and (\ref{sdeee}) that require more than a passing mention are $\Delta_{\mathcal{C}_{1}}\arg \zeta(s), \Delta_{\mathcal{C}}\arg\Gamma\left(\frac{s}{2}\right)$ and $\Delta_{\mathcal{C}_{2}}\arg a(s)$. For the first use
\begin{equation*}
|\arg \zeta(\sigma_{1} + it)| \leq |\log \zeta(\sigma_{1} + it)| \leq \log\zeta(\sigma_{1}),
\end{equation*}
and for the second use
\begin{equation*}\label{Stirling}
\log\Gamma(z) = \left(z-\frac{1}{2}\right)\log z - z + \frac{1}{2} \log 2\pi + \frac{\theta}{6|z|},
\end{equation*}
(see, e.g., \cite[p.\ 294]{Olver}) which is valid for $|\arg z| \leq \frac{\pi}{2}$, and in which $\theta$ denotes a complex number satisfying $|\theta| \leq 1$. To estimate $\Delta_{\mathcal{C}_{2}}\arg a(s)$ write
\begin{equation}\label{wwbw}
f(s) = \frac{1}{2} \{a(s+ iT)^{N} + a(s-iT)^{N}\},
\end{equation}
for some positive integer $N$, to be determined later. Thus $f(\sigma) = \Re a(\sigma + iT)^{N}$. Suppose that there are $n$ zeroes of $\Re a(\sigma + iT)^{N}$ for $\sigma\in \mathcal{C}_{2}$. These zeroes partition the segment into $n+1$ intervals. On each interval $\arg a(\sigma + iT)^{N}$ can increase by at most $\pi$, whence
\begin{equation*}
|\Delta_{\mathcal{C}_{2}} \arg a(s)| = \frac{1}{N} |\Delta_{\mathcal{C}_{2}} \arg a(s)^{N}| \leq \frac{(n+1)\pi}{N}
\end{equation*}
In conclusion, when $T\geq 1$
\begin{equation}\label{3.11}
\begin{split}
\bigg|N(T) - \frac{T}{2\pi}\log\frac{T}{2\pi e } -\frac{7}{8}\bigg| &\leq \frac{1}{4\pi}\tan^{-1}\frac{1}{2T} + \frac{T}{4\pi}\log\left(1 + \frac{1}{4T^2}\right) + \frac{1}{3\pi T} + |S(T)|\\
&\leq \frac{0.2}{T} + |S(T)|,
\end{split}
\end{equation}
where
\begin{equation*}\label{3.333}
\begin{split}
|S(T)| &\leq \frac{(n+1)}{N} + \frac{1}{\pi}\log\zeta(\sigma_{1}) + \frac{1}{\pi}\tan^{-1}\frac{1}{2T} + \frac{1}{\pi}\tan^{-1}\frac{\sigma_{1}}{T}\\
&\leq \frac{(n+1)}{N} + \frac{1}{\pi}\log\zeta(\sigma_{1}) + \frac{0.8}{T}.
\end{split}
\end{equation*}
The inequality in (\ref{3.11}) enables one to deduce Corollary \ref{onlycor} from Theorem \ref{Onenad}.

\section{Estimating $n$}
One may estimate $n$ with Jensen's Formula.
\begin{Lem}[Jensen's Formula]\label{JL}
Let $f(z)$ be holomorphic for $|z-a|\leq R$ and non-vanishing at $z=a$. Let the zeroes of $f(z)$ inside the circle be $z_{k}$, where $k=1, 2, \ldots, n$, and let $|z_{k} - a| = r_{k}$. Then
\begin{equation}\label{JF1}
\log\frac{R^{n}}{|r_{1} r_{2} \cdots r_{n}|} = \frac{1}{2\pi}\int_{0}^{2\pi} \log f(a+ Re^{i\phi})\, d\phi - \log |f(a)|.
\end{equation}
\end{Lem}

The following lemma, proved in \cite{TrudDed}, is now used to invoke Backlund's trick. For a complex-valued function $f(s)$, define $\Delta_{\pm}\arg f(s)$ to be the change in argument of $f(s)$ as $\sigma$ varies from $\frac{1}{2}$ to $\frac{1}{2} \pm \delta$, where $\delta>0$.

\begin{Lem}\label{BL}\textbf{(i)}
Let $N$ be a positive integer and let $T\geq T_{0}\geq 1$. Suppose that
\begin{equation*}
| \Delta_{+}\arg a(s) + \Delta_{-}\arg a(s)| < E,
\end{equation*}
where $E = E(\delta, T_{0})$. If there are $n$ zeroes of $\Re a(\sigma + iT)^{N}$ for $\sigma \in [\frac{1}{2}, \sigma_{1}]$, then there are at least $n-1 - [N E/\pi]$ zeroes in $\sigma \in [1-\sigma_{1}, \frac{1}{2}]$.

\textbf{(ii)}
Denote the zeroes in $[\frac{1}{2}, \sigma_{1}]$ by $\rho_{\nu} = a_{\nu} + iT$ where $\frac{1}{2} \leq a_{n} \leq a_{n-1} \leq \cdots \leq \sigma_{1}$, and the zeroes in $[1-\sigma_{1}, \frac{1}{2}]$ by $\rho_{\nu}' = a_{\nu}' + iT$ where $1-\sigma_{1} \leq a_{1}' \leq a_{2}' \leq \cdots \leq \frac{1}{2}$. Then
\begin{equation*}\label{jenpre}
a_{\nu} \geq 1-a_{\nu}', \quad \textrm{for } \nu = 1, 2, \ldots, n-1 - [NE/\pi],
\end{equation*}
and, if $\sigma_{1} = \frac{1}{2} + \sqrt{2}(\eta + \frac{1}{2})$, then
\begin{equation}\label{finalae}
\prod_{\nu = 1}^{n} |1 + \eta - a_{\nu}| \prod_{\nu = 1}^{n -1 - [NE/\pi]} |1 + \eta - a_{\nu}'|\leq (\tfrac{1}{2} + \eta)^{2n-1 - [NE/\pi]}.
\end{equation}
\end{Lem}
\begin{proof}
This was proved in \cite[Lemma 2]{TrudDed} for Dirichlet $L$-functions; the proof for $a(s) = (s-1)\zeta(s)$ is identical.
\end{proof}
\subsection{Calculation of $E$}
One may use \cite[(5.4)]{TrudDed} to estimate $E$ in Lemma \ref{BL}. 
\begin{equation*}\label{Gdef}
\begin{split}
E\leq G(\delta,T) =& (-\frac{5}{4} + \frac{\delta}{2}) \tan^{-1}\frac{\frac{1}{2} + \delta}{T} - (\frac{5}{4} + \frac{\delta}{2}) \tan^{-1}\frac{\frac{1}{2} - \delta}{T}  + \frac{5}{4} \tan^{-1} \frac{1}{2T}\\
&\quad -\frac{T}{4}\log\left[1+ \frac{2\delta^{2}(T^{2} - \frac{1}{4}) + \delta^{4}}{\left(T^{2} + \frac{1}{4}\right)^{2}}\right]
+\frac{4\theta}{3T}.
\end{split}
\end{equation*}
One can show that $G(\delta, T)$ is decreasing in $T$ and increasing in $\delta$. Therefore, since, in Lemma \ref{BL} (i) one takes $\sigma_{1} = \frac{1}{2} + \sqrt{2}(\frac{1}{2} + \eta)$, it follows that $\delta \leq \sqrt{2}(\frac{1}{2} + \eta)$, whence 
\begin{equation*}\label{Etake}
E \leq  G(\sqrt{2}, T_{0}) \leq \frac{4.4}{T_{0}},
\end{equation*}
for $T\geq T_{0}$.

\subsection{Applying Jensen's Formula}

In Lemma \ref{JL}, take $a = 1+\eta$, $f(z)$ as in (\ref{wwbw}), and $R = r(\frac{1}{2}+ \eta)$, where $r>1$. Suppose that there are $n$ zeroes of $\Re a(\sigma+ iT)^{N}$ for $\sigma \in [\frac{1}{2}, \sigma_{1}]$, where $\sigma_{1} = \frac{1}{2} + \sqrt{2}(\eta + \frac{1}{2})$. Initially, to take advantage of Backlund's trick, one needs $1 + \eta - r(\frac{1}{2} + \eta) \leq 1-\sigma_{1}$, so that all of the zeroes are included in the contour. The argument in \cite[\S 4.1]{TrudDed} shows that one can use any $r>1$. The following results are simplified greatly if one imposes an upper bound on $r$. Indeed, to use (\ref{Gestimate}) requires $r(\frac{1}{2} + \eta) \leq \frac{3}{2} + \eta \leq 2$.

To apply Jensen's formula it is necessary to show that $f(1+\eta)$ is non-zero: this is easy to do upon invoking an observation due to Rosser \cite{Rossers}. Write $a(1+\eta + iT) = Ke^{i\psi}$, where $K> 0$. Choose a sequence of $N$'s tending to infinity for which $N\psi$ tends to zero modulo $2\pi$. Thus
\begin{equation}\label{Rost}
\frac{f(1+\eta)}{|a(1+\eta+iT)|^{N}} \rightarrow 1.
\end{equation}

It follows from (\ref{JF1}) and (\ref{finalae}) that
\begin{equation*}\label{nbound}
n \leq \frac{1}{4\pi\log r} J - \frac{1}{2\log r} \log |f(1+\eta)| +\frac{1}{2} + \frac{NE}{2\pi},
\end{equation*}
where
\begin{equation*}\label{71}
J = \int_{-\frac{\pi}{2}}^{\frac{3\pi}{2}} \log|f(1+\eta + r(\tfrac{1}{2} + \eta)e^{i\phi})|\, d\phi.
\end{equation*}
First one may bound $\log|f(1+ \eta)|$ using (\ref{Rost}) and the trivial bound $|\zeta(s)| \geq \frac{\zeta(2\sigma)}{\zeta(\sigma)}$. Thus\begin{equation*}\label{f0}
\log |f(1+\eta)| =  N\log |a(1+ \eta +iT)| + o(1) \geq N\left(\log  T + \log\frac{\zeta(2+2\eta)}{\zeta(1+ \eta)}\right) + o(1).
\end{equation*}

\section{Estimating the integrals}
Divide $J$ into five pieces thus
\begin{equation*}
\begin{split}
R_{0}&=  \{ s: 1+ \eta\leq \sigma \leq 1+ \eta + r(\tfrac{1}{2} + \eta)\}  =\{\phi: -\tfrac{\pi}{2} \leq \phi \leq \tfrac{\pi}{2}\},\\
R_{1} &= \{ s: 1\leq \sigma \leq 1+ \eta\; \textrm{and } t\geq T\} = \{\phi: \tfrac{\pi}{2} \leq \phi \leq \tfrac{\pi}{2}+ \phi_{1}\}, \\
R_{2} &= \{s: \tfrac{1}{2}\leq \sigma \leq 1\; \textrm{and } t\geq T\}= \{\phi: \tfrac{\pi}{2} + \phi_{1} \leq \phi \leq \tfrac{\pi}{2} + \phi_{2}\}, \\
R_{3} &= \{ s: 0\leq \sigma \leq \tfrac{1}{2}\; \textrm{and } t\geq T\}= \{\phi: \tfrac{\pi}{2} + \phi_{2} \leq \phi \leq \tfrac{\pi}{2} + \phi_{3}\},  \\
R_{4} &= \{ s: 1+ \eta - r(\tfrac{1}{2} + \eta)\leq \sigma \leq 0\; \textrm{and } t\geq T\} = \{\phi: \tfrac{\pi}{2} + \phi_{3} \leq \phi \leq \pi\},  \\
 \end{split}
\end{equation*}
where
\begin{equation*}
\phi_{1} = \sin^{-1} \frac{\eta}{r(\frac{1}{2} + \eta)}, \quad  \phi_{2} = \sin^{-1} \frac{1}{r}, \quad \phi_{3} = \sin^{-1} \frac{1+\eta}{r(\frac{1}{2} + \eta)}.
 \end{equation*}

One may then write $J = \int_{R_{0}} + 2 \int_{R_{1}} + 2 \int_{R_{2}} + 2 \int_{R_{3}} + 2 \int_{R_{4}}.$ In estimating each integral some small error terms labelled $\epsilon_{0}, \ldots, \epsilon_{4}$ are encountered. Since these are all $O(T_{0}^{-1})$, they have been estimated with a great deal of alacrity. It is neither essential nor insightful to strive for the smallest bounds on these terms.

\subsection{Convexity bounds}
To estimate $\zeta(s)$ on $R_{0}$ one may use the trivial estimate $|\zeta(s)| \leq \zeta(\sigma)$. On $R_{1}, \ldots, R_{4}$ one can use the following version of the Phragm\'{e}n--Lindel\"{o}f principle.
\begin{Lem}\label{lem1}
Let $a, b, Q$ and $k$ be real numbers, such that $Q+a>1$, and let $ f(s)$ be regular analytic in the strip $\leq a\leq \sigma \leq b$ and satisfy the growth condition 
$$ |f(s)| <C\exp \left\{e^{k|t|}\right\},$$ for a certain $C>0$ and for $0<k<\pi/(b-a)$. Also assume that
\[|f(s)|\leq\left\{\begin{array}{ll}
A|Q+s|^{\alpha_{1}} (\log|Q+ s|)^{\alpha_{2}} &  \mbox{for $\Re(s) = a,$}\\
B|Q+s|^{\beta_{1}}(\log |Q + s|)^{\beta_{2}} & \mbox{for $\Re(s) = b,$}\\
\end{array}
\right.\]
where $\alpha_{1} \geq \beta_{1}$ and where $\alpha_{1}, \alpha_{2}, \beta_{1}, \beta_{2} \geq 0$. Then throughout the strip $a \leq \sigma \leq b$ the following holds
$$ |f(s)| \leq \{A|Q+s|^{\alpha_{1}}|\log(Q+s)|^{\alpha_{2}}\}^{\frac{b-\sigma}{b-a}}  \{B|Q+s|^{\beta_{1}}|\log(Q+s)|^{\beta_{2}}\}^{\frac{\sigma-a}{b-a}}.$$
\end{Lem}
\begin{proof}
This extends a result due to Rademacher \cite[pp.\ 66-67]{Rad} so as to incorporate logarithms. Form the function 
\begin{equation*}
F(s) = f(s) \phi(s;Q)E^{-1}e^{-vs} \left\{\log (Q+s)\right\}^{\frac{\alpha_{2}(\sigma - b) + \beta_{2}(a-\sigma)}{b-a}},
\end{equation*}
where $\phi(s; Q)$ is the function of \cite[Theorem 1]{Rad}, and $E$ and $\nu$ are determined by $A = E e^{\nu a}$ and\footnote{Note that in \cite[(3.7)]{Rad} there is a typo: $B= e^{\nu b}$ should read $B= Ee^{\nu b}$.}  $B = Ee^{\nu b}$. Since $Q+a>1$, the function $F(s)$ is holomorphic in the strip $a \leq \sigma \leq b$. The proof now proceeds as in \cite{Rad}.
\end{proof}

Lemma \ref{lem1} will be applied to $R_{1}, \ldots, R_{4}$ where it will be convenient to write $|Q + s|$ in terms of $T$. If $Q = Q_{0}\leq 1000$ and $T\geq T_{0}\geq 10^{6}$ one may write
\begin{equation}\label{schle}
|\log|Q_{0} + s| - \log T| \leq \frac{6}{T_{0}} \leq 10^{-5}.
\end{equation}
If, in addition, $Q_{0}\geq 1$, then $|\arg (Q_{0} + s) |\leq \frac{\pi}{2}$ on $R_{1}, \ldots, R_{4}$. Using this, (\ref{schle}), and the identity
\begin{equation*}
|\log z| = |\log |z|| \left\{ 1 + \left(\frac{\arg z}{\log |z|}\right)^{2}\right\}^{1/2},
\end{equation*}
one deduces that
\begin{equation}\label{schle2}
|\log (Q_{0} + s)|\leq 1.007\log T, \quad\quad \log|\log (Q_{0} + s)| \leq\log\log T + 0.007.
\end{equation}

\subsection{$R_{0}$}\label{Trivial}

On $R_{0}$
\begin{equation*}
|a(s)| \leq |\eta + r(\tfrac{1}{2} + \eta)e^{i\phi} \pm iT|\, \zeta(1 + \eta + r(\tfrac{1}{2} + \eta)\cos \phi),
 \end{equation*}
whence 
\begin{equation*}\label{Ieq}
\frac{R_{0}}{N} \leq \pi (\log T + \epsilon_{0}) + \int_{-\pi/2}^{\pi/2} \log\zeta(1+ \eta + r(\tfrac{1}{2} + \eta)\cos\phi)\, d\phi,
\end{equation*}
where
\begin{equation*}\label{epsdef}
\epsilon_{0} = \frac{r(\frac{1}{2} + \eta)}{T_{0}} + \frac{\{\eta + r(\frac{1}{2} + \eta)\}^{2}}{2T_{0}^2} \leq \frac{2}{T_{0}} + \frac{25}{8T_{0}^{2}}\leq \frac{3}{T_{0}}.
\end{equation*}

\subsection{$R_{1}$}
On $\sigma = 1+ \eta$ bound $\zeta(\sigma)$ trivially, whence, for any $Q_{0} \geq 0$
\begin{equation}\label{bound4}
|a(1+ \eta + it)| \leq |Q_{0} + (1+\eta + it)| \zeta(1+ \eta).
\end{equation}

On $\sigma = 1$ one may make use of Backlund's estimate \cite[(53)]{Backlund1918} that 
\begin{equation}\label{Best}
|\zeta(1+ it)|\leq \log t,
\end{equation}
for $t\geq 50$. This, and a computation check for small $t$ shows that
\begin{equation}\label{bound3}
|a( 1+ it)| \leq |Q_{0} + (1 + it)| \log|Q_{0} + (1+ it)|,
\end{equation}
for all $t$ and for any $Q_{0}\geq 1$. It follows from Lemma \ref{lem1}, (\ref{schle}), (\ref{schle2}), (\ref{bound4}) and (\ref{bound3}) that
\begin{equation*}
\begin{split}
\frac{R_{1}}{N} &\leq (\log T + 10^{-5})\phi_{1} + \log\zeta(1+ \eta) \left[ \phi_{1} - \frac{r(\frac{1}{2} + \eta)(1 - \cos\phi_{1})}{\eta}\right] \\
&+ (\log\log T + 0.007) \left[\frac{r(\frac{1}{2} + \eta)(1- \cos\phi_{1})}{\eta}\right].
\end{split}
\end{equation*}

\subsection{$R_{2}$}
Suppose that one is equipped with a bound 
\begin{equation*}\label{zeta12}
|\zeta(\tfrac{1}{2} + it)|\leq k_{1} t^{k_{2}} (\log t)^{k_{3}}, \quad \textrm{for } t\geq t_{0},
\end{equation*}
in which $0\leq k_{3}\leq 10$, say. This upper bound on $k_{3}$ is imposed merely to simplify the resulting error term. The convexity bound for $\zeta(s)$ shows that $k_{2}\leq \frac{1}{4}$.

It follows that
\begin{equation}\label{bound2}
|a(\tfrac{1}{2} + it)| \leq k_{1}|Q_{0} + (\tfrac{1}{2} + it)|^{k_{2} + 1} (\log|Q_{0} + (\tfrac{1}{2} + it)|)^{k_{3}}, \quad \textrm{for } t\geq t_{0},
\end{equation}
for any $Q_{0}\geq 0$. It is always possible to choose $Q_{0}$ large enough so that (\ref{bound2}) holds for all $t$. It follows from Lemma \ref{lem1}, (\ref{schle}), (\ref{schle2}), (\ref{bound3}) and (\ref{bound2}) that
\begin{equation*}
\begin{split}
\frac{R_{2}}{N} &\leq (\log T + 10^{-5})\left[ 2k_{2}r(\tfrac{1}{2} + \eta)(\cos\phi_{1} - \cos\phi_{2}) + (\phi_{2} - \phi_{1})(1- 2k_{2}\eta)\right] \\
&+ 2\log k_{1} \left[r(\tfrac{1}{2} + \eta)(\cos\phi_{1} - \cos\phi_{2}) - \eta(\phi_{2} - \phi_{1})\right]\\
&+ 2(\log\log T + 0.007)[ r(\tfrac{1}{2} +\eta)(1- k_{3})(\cos\phi_{2} - \cos\phi_{1})\\
&\qquad\qquad\qquad\qquad\quad\quad+ (\phi_{2} - \phi_{1})(\tfrac{1}{2} + \eta - k_{3}\eta)].\\
\end{split}
\end{equation*}

\subsection{$R_{3}$}
The functional equation
\begin{equation}\label{functe}
\zeta(s) = \pi^{s-\frac{1}{2}} \frac{\Gamma(\frac{1}{2} - \frac{1}{2}s)}{\Gamma(\frac{1}{2}s)} \zeta(1-s),
\end{equation}
(see, e.g., \cite[Ch.\ II]{Titchmarsh}), the estimate
\begin{equation}\label{Gestimate}
\Bigg|\frac{\Gamma(\frac{1}{2} - \frac{1}{2}s)}{\Gamma(\frac{1}{2}s)}\Bigg| \leq \left(\frac{|1+s|}{2}\right)^{\frac{1}{2}-\sigma},
\end{equation}
for $-\frac{1}{2} \leq \sigma \leq \frac{1}{2}$ (see \cite[p.\ 197]{Rademacher}), and Backlund's estimate (\ref{Best}) show that
\begin{equation*}
|\zeta(it)|\leq \left(\frac{|1 + it|}{2\pi}\right)^{\frac{1}{2}}\log t, \quad\textrm{for } t\geq 50.
\end{equation*}
A computational check shows that
\begin{equation}\label{bound1}
|a(it)|\leq (2\pi)^{-\frac{1}{2}} |Q_{0} + it|^{\frac{3}{2}} \log|Q_{0} + it|,
\end{equation}
for all $t$ and for any $Q_{0}\geq 2$. It follows from Lemma \ref{lem1}, (\ref{schle}), (\ref{schle2}), (\ref{bound2}) and (\ref{bound1}) that
\begin{equation*}
\begin{split}
\frac{R_{3}}{N}\leq & (\log T + 10^{-5})[r(\tfrac{1}{2} + \eta)(1-2k_{2})(\cos\phi_{2} - \cos\phi_{3}) \\
& \qquad\qquad\qquad\quad+ (\phi_{3} - \phi_{2})\{2k_{2}(1+ \eta) + \tfrac{1}{2} - \eta\}]\\
&+ 2(\log\log T + 0.007)[(k_{3} - 1)r(\tfrac{1}{2} + \eta)(\cos\phi_{3} - \cos\phi_{2}) \\
& \qquad\qquad\qquad\qquad\qquad+ (\phi_{3} - \phi_{2})\{k_{3}(1 + \eta) - (\tfrac{1}{2} + \eta)\}]\\
&+ \log 2\pi [(\tfrac{1}{2} + \eta)(\phi_{3} - \phi_{2}) - r(\tfrac{1}{2} + \eta)(\cos\phi_{2} - \cos\phi_{3})]\\
& + 2\log k_{1}[(1+ \eta)(\phi_{3} - \phi_{2}) - r(\tfrac{1}{2} + \eta)(\cos\phi_{2} - \cos\phi_{3})].
\end{split}
\end{equation*}

\subsection{$R_{4}$}
(\ref{functe}), (\ref{Gestimate}), and the trivial estimate for $\zeta(s)$ show that
\begin{equation}\label{bound0}
|a(q + it)| \leq (2\pi)^{q-\frac{1}{2}} |Q_{0} + (q + it)|^{\frac{3}{2} - q} \zeta(1- q),
\end{equation}
for all $t$ and for all $Q_{0}\geq 2$, where $q = 1+ \eta - r(\frac{1}{2} + \eta)$.

Lemma \ref{lem1},  (\ref{schle}), (\ref{schle2}), (\ref{bound1}) and (\ref{bound0}) show that
\begin{equation*}
\begin{split}
\frac{R_{4}}{N}&\leq (\log T + 10^{-5})[(\tfrac{1}{2} - \eta)(\tfrac{\pi}{2} - \phi_{3}) + r(\tfrac{1}{2} + \eta)] \\
&+ (\log\log T + 0.007)\left[\frac{r(\tfrac{1}{2} +\eta)(\frac{\pi}{2} - \phi_{3} - \cos\phi_{3})}{r(\tfrac{1}{2} + \eta) - (1+ \eta)}\right]\\
& + \log 2\pi [(\tfrac{1}{2} + \eta)(\tfrac{\pi}{2} - \phi_{3}) - r(\tfrac{1}{2} + \eta)\cos\phi_{3}] \\
&+ \log \zeta(r(\tfrac{1}{2} + \eta) - \eta) \bigg[\frac{r(\tfrac{1}{2} + \eta) - (\tfrac{\pi}{2} - \phi_{3})(1+ \eta)}{r(\tfrac{1}{2} + \eta) - (1+ \eta)}\bigg].
\end{split}
\end{equation*}
One therefore has all of the results needed to bound $|S(t)|$. This produces an expression the inelegance of which prohibits its being inserted in this paper. The next section will provide some specific information.

\section{Specific values of $k_{1}, k_{2}$ and $k_{3}$}
Cheng and Graham \cite[Thm.\ 3]{Cheng} proved that
\begin{equation}\label{CGT3}
|\zeta(\tfrac{1}{2} + it)| \leq 1.457t^{1/6}\log t + 40.995t^{1/6} + 1.863\log t + 123.125,
\end{equation}
and that
\begin{equation}\label{CGT2}
|\zeta(\tfrac{1}{2} + it)| \leq 6t^{1/4} + 41.129,
\end{equation}
where both (\ref{CGT3}) and (\ref{CGT2}) are valid for $t\geq e$. They combine these results with a computational check to show that
\begin{equation}\label{CGTconc}
|\zeta(\tfrac{1}{2} + it)|\leq 3t^{1/6}\log t,
\end{equation}
for $t\geq e$. Actually, their proof enables one to write 2.76 in place of 3 in (\ref{CGTconc}). This can be improved by combining (\ref{CGT3}) not with (\ref{CGT2}) but with the estimate
\begin{equation*}\label{CGT2rep}
|\zeta(\tfrac{1}{2} + it)| \leq \frac{4}{(2\pi)^{\frac{1}{4}}} t^{1/4}, 
\end{equation*}
(see, e.g., \cite[Lemma 2]{Lehman}) which is valid for $t\geq 1$. This shows that
\begin{equation*}\label{Boundonhalf}
|\zeta(\tfrac{1}{2} +it)|\leq 2.38 t^{1/6}\log t,
\end{equation*}
for $t\geq e$. This, and a small computational check for $0\leq t \leq e$, shows that (\ref{bound2}) holds with $k_{1} = 2.38$, $k_{2}= \frac{1}{6}$, $k_{3} = 1$ and $Q_{0} \geq 1$.

One can now take $Q_{0} =2$, whence all of (\ref{bound4}), (\ref{bound3}), (\ref{bound2}), (\ref{bound1}), (\ref{bound0}) are satisfied.

\section{Conclusion}
Combining the above results shows that, when $T\geq T_{0}\geq 1,$
\begin{equation}\label{Sfinal}
|S(T)| \leq a \log T + b\log\log T + c,
\end{equation}
where 
\begin{equation*}
a= \frac{\phi_{1}\eta + (\phi_{2} - \frac{3\pi}{2})(\frac{1}{2} + \eta) + \phi_{3}(1 + \eta) + r(\tfrac{1}{2} + \eta)(\cos\phi_{1} + \cos\phi_{2} + \cos\phi_{3})}{6\pi \log r},
\end{equation*}
\begin{equation*}
b= \frac{-\phi_{1}+ \phi_{3}  + r(\tfrac{1}{2} + \eta)\left\{\frac{1-\cos\phi_{1}}{\eta} + \frac{\tfrac{\pi}{2} - \cos\phi_{3} - \phi_{3}}{r(\frac{1}{2} + \eta) -(1+ \eta)}\right\}}{2\pi \log r},
\end{equation*}
\begin{equation*}
\begin{split}
c&=  \frac{\log\frac{\zeta(1+ \eta)}{\zeta(2 + 2\eta)}}{2\log r} + \frac{1}{\pi}\log\zeta(\tfrac{1}{2} + \sqrt{2}(\eta + \tfrac{1}{2})) + \frac{\int_{-\pi/2}^{\pi/2} \log\zeta(1+ \eta + r(\tfrac{1}{2} + \eta)\cos\phi)\, d\phi}{4\pi\log r}\\
&+ \bigg\{\log\zeta(1+ \eta)\left[\phi_{1} + \frac{r(\frac{1}{2} + \eta)(\cos\phi_{1} - 1)}{\eta}\right] + (\tfrac{1}{2} + \eta)(\tfrac{\pi}{2} - \phi_{2} - r\cos\phi_{2})\log 2\pi \\
&-2\log k_{1} [r(\tfrac{1}{2} + \eta)(2\cos\phi_{2} - \cos\phi_{1} - \cos\phi_{3}) + \phi_{2} - \phi_{3} + \eta(2\phi_{2} - \phi_{1} - \phi_{3})]\\
&+ \left[\frac{(1+ \eta)(\frac{\pi}{2} - \phi_{3}) - r(\frac{1}{2} + \eta)\cos\phi_{3}}{1+ \eta - r(\frac{1}{2} + \eta)}\right]\log\zeta[r(\tfrac{1}{2} + \eta) - \eta]\bigg\}\bigg/(2\pi \log r) + 0.003.
\end{split}
\end{equation*}
Taking $\eta = 0.06$ and $r = 2.08$ proves Theorem \ref{Onenad} for $T\geq 6.8\cdot 10^{6}$. When $T< 6.8 \cdot 10^{6}$, Theorem \ref{Onenad} follows from (\ref{Odl}).

\section{Improvements}
Theorem \ref{Onenad} is improved instantly if one can provide better bounds for the growth of $|\zeta(\tfrac{1}{2} + it)|$ and $|\zeta(1 + it)|$. For the former one seeks an explicit bound on 
\begin{equation}\label{lastaa}
|\zeta(\tfrac{1}{2} + it)| \ll t^{\theta},
\end{equation}
for some $\theta \leq \frac{1}{6}$. The easiest approach seems to be to make explicit the result of Titchmarsh \cite[Thm.\ 5.18]{Titchmarsh} which has $\theta = \frac{27}{164}$ in (\ref{lastaa}). 
For the latter one could make explicit the estimate $\zeta(1 + it) = O(\log t/\log\log t)$, by following the arguments preparatory to proving Theorem 5.16 in \cite{Titchmarsh}. 

Both of these approaches are being investigated by the author.

\section*{Acknowledgements}
I wish to thank Yannick Saouter, Nuno Costa Pereira, Olivier Ramar\'{e} for their encouragement.
\bibliographystyle{plain}
\bibliography{themastercanada}

\end{document}